\documentclass{amsart}

\usepackage{amsthm}
\usepackage{cite}
\usepackage[dvipdfmx]{graphicx}

\theoremstyle{definition}
\newtheorem{theorem}{Theorem}[section]
\newtheorem{lemma}[theorem]{Lemma}
\newtheorem{proposition}[theorem]{Proposition}
\newtheorem{fact}[theorem]{Fact} 
\newtheorem{observation}[theorem]{Observation} 

\theoremstyle{definition}
\newtheorem{definition}[theorem]{Definition}
\newtheorem{example}[theorem]{Example}

\newtheorem*{ac}{Acknowledgments} 

\theoremstyle{remark}
\newtheorem{remark}[theorem]{Remark}

\newenvironment{rmenum}{
\begin{enumerate}

}
{\end{enumerate}}

\title[Parity Factors I]{Parity Factors I: General Kotzig-Lov\'asz  Decomposition for Grafts}
\author{Nanao Kita}
\address{National Institute of Informatics
2-1-2 Hitotsubashi, Chiyoda-ku, Tokyo, Japan 101-8430}
\email{kita@nii.ac.jp}

\date{\today}

\newcommand{\distgt}[4]{\lambda(#3, #4; #1, #2)}

\newcommand{\distgtf}[5]{\lambda(#4, #5; #3; #1, #2)}

\newcommand{\parcut}[2]{\delta_{#1}(#2)}

\newcommand{\tcomp}[2]{\mathcal{G}(#1, #2)}

\newcommand{\levelr}[2]{U_{#1}(#2)}

\newcommand{\laysigr}[2]{U_{#1}(#2)}

\newcommand{\layler}[2]{U_{ \le #1}(#2)}

\newcommand{\conn}[1]{\mathcal{C}(#1)}

\newcommand{\gsim}[1]{\sim_{#1}}
\newcommand{\gtsim}[2]{\sim_{(#1, #2)}}
\newcommand{\gtpart}[2]{\mathcal{P}(#1, #2)}
\newcommand{\pargtpart}[3]{\mathcal{P}(#3; #1, #2)}

\begin{document}
\maketitle

\begin{abstract} 
This paper is the first  from a series of  papers 
that establish a generalization of the basilica decomposition for cardinality minimum joins in grafts. 
Joins in grafts are also known as $T$-joins in graphs, 
where $T$ is a given set of vertices, 
and minimum joins in grafts can be considered as a generalization 
of perfect matchings in graphs provided in terms of parity. 
The basilica decomposition is a canonical decomposition applicable to 
general graphs with perfect matchings, 
and the  general Kotzig-Lov\'asz decomposition is one of the three central concepts that compose this theory.  
The classical Kotzig-Lov\'asz decomposition is a canonical decomposition 
 for a special class of graphs known as {\em factor-connected graphs}   and is famous for its contribution to the study of the matching polytope and lattice.  
The general Kotzig-Lov\'asz decomposition is a nontrivial generalization 
of its classical counterpart and is applicable to general graphs with perfect matchings.  
As a component of the basilica decomposition theory, 
the general Kotzig-Lov\'asz decomposition has  contributed 
to the derivation of further results in matching theory, 
such as a characterization of barriers or an alternative proof of the tight cut lemma. 
In this paper, we present an analogue of the general Kotzig-Lov\'asz decomposition 
for minimum joins in grafts. 
\end{abstract}

\section{Introduction} 
This paper is the first  in a series of papers~\cite{kitatcath, kitatcathalg}
that establish the basilica decomposition~\cite{kitacathedral, DBLP:journals/corr/abs-1205-3816, DBLP:conf/isaac/Kita12} for grafts. 
In this paper, 
we present the general Kotzig-Lov\'asz decomposition~\cite{kotzig1959a, kotzig1959b, kotzig1960, lovasz1972b, kitacathedral, DBLP:journals/corr/abs-1205-3816, DBLP:conf/isaac/Kita12} for grafts.

The general Kotzig-Lov\'asz decomposition is a canonical decomposition 
of graphs that describes the structure of perfect matchings. 
In the theory of matchings ($1$-matchings), 
{\em canonical decompositions}, such as 
the {\em Gallai-Edmonds}~\cite{edmonds1965, gallai1964} or {\em Dulmage-Mendelsohn}~\cite{dm1958, dm1959, dm1963} decompositions, 
are fundamental tools that  constitute the basis of the theory~\cite{lp1986}. 
In matching theory, something is said to be {\em canonical} if it is determined uniquely for a given graph.  
A canonical decomposition is a decomposition determined uniquely for a graph 
and thus provides us with a comprehensive view of the structure of all maximum matchings. 
Canonical decompositions are therefore strong tools for analyzing matchings.

The classical Kotizg-Lov\'asz decomposition~\cite{kotzig1959a, kotzig1959b, kotzig1960, lovasz1972b} 
  was proposed for a particular class of graphs with perfect matchings 
known as {\em factor-connected graphs}. 
This class of graphs is essential to  polyhedral studies of perfect matchings~\cite{lp1986, schrijver2003}. 
The classical Kotzig-Lov\'asz decomposition was used for deriving 
 important combinatorial results for factor-connected graphs, 
such as the {\em two ear theorem} and {\em tight cut lemma}~\cite{lp1986}, 
both of which, along with the classical Kotzig-Lov\'asz decomposition itself, 
 have been  used for deriving central results regarding the perfect matching polytope and lattice~\cite{schrijver2003, lp1986}.

Recently, 
the classical Kotzig-Lov\'asz decomposition was generalized 
for general graphs with perfect matchings~\cite{kitacathedral, DBLP:journals/corr/abs-1205-3816, DBLP:conf/isaac/Kita12}.\footnote{This is not a trivial generalization;  
the general Kotzig-Lov\'asz decomposition is not the mere disjoint union of the classical Kotzig-Lov\'asz decomposition of each factor-component  but is a refinement of it.} 
We call this the {\em general Kotzig-Lov\'asz decomposition} 
or just the {\em Kotzig-Lov\'asz decomposition}.   
The general Kotzig-Lov\'asz decomposition is proposed as 
one of the three main components of the theory of {\em basilica decomposition}~\cite{kitacathedral, DBLP:journals/corr/abs-1205-3816, DBLP:conf/isaac/Kita12}. 
The basilica decomposition is a canonical decomposition 
applicable to general graphs with perfect matchings 
and consists of the following three main concepts: 
\begin{rmenum} 
\item \label{item:kl} 
the general Kotzig-Lov\'asz decomposition,  
\item \label{item:order} 
the {\em cathedral order}, 
a canonical partial order between {\em factor-components}, and 
\item \label{item:rel} 
the canonically determined relationship between \ref{item:kl} and \ref{item:order}. 
\end{rmenum} 
The basilica decomposition 
has been used for deriving alternative proofs~\cite{kita2014alternative, kita2015graph}
of classical theorems such as {\em Lov\'asz's cathedral theorem for saturated graphs}~\cite{lovasz1972b} 
and Edmonds et al.'s {\em tight cut lemma}~\cite{elp1982}. 
It is also used for providing 
a characterization of {\em barriers}~\cite{lp1986, kita2012canonical, kita2013disclosing} 
and a generalization of the {\em Dulmage-Mendelsohn decomposition}~\cite{dm1958, dm1959, dm1963}, 
 for arbitrary graphs~\cite{kitagdm}, which was originally for bipartite graphs.

Perfect matchings can be generalized into {\em $T$-joins}~\cite{lp1986, schrijver2003}.  
Given a graph and a set of vertices $T$, 
a $T$-join is a set of edges $F$ 
such that each vertex is adjacent to an odd number of edges from $F$ 
if and only if it is a vertex from $T$. 
A graph has a $T$-join if and only if each connected component has an even number of vertices from $T$. 
The concept of $T$-joins in graphs 
are also known as {\em joins} in {\em grafts}; 
a graft is defined as a tuple of graph and set $T$ that can possess a $T$-join. 
In this paper, we often use this terminology of joins and grafts. 
Even in a graph with $T$-joins, 
the minimum $T$-join problem, 
namely, 
which set of edges can be a $T$-join with the minimum number of edges,  is not trivial. 
If a given graph has a perfect matching and $T$ is equal to its vertex set, 
then minimum $T$-joins of this graph coincide with perfect matchings. 
That is, minimum $T$-joins are a generalization of perfect matchings 
that is provided in terms of  parity. 
The minimum $T$-join problem  includes  well-known classical  problems 
 such as the {\em Chinese postman problem}~\cite{lp1986, schrijver2003}. 
Additionally, 
$T$-joins are known to be closely related to famous open problems in graph theory
 such as Tutte's $4$-flow conjecture~\cite{schrijver2003}.

In this paper, we present a generalization of the general Kotzig-Lov\'asz decomposition for $T$-joins. 
This result is, in fact, a component of the entire theory of  the basilica decomposition  
for $T$-joins; 
 in this paper and its sequels~\cite{kitatcath, kitatcathalg}, 
 the three central concepts of the basilica decomposition theory are generalized for $T$-joins. 
Our result in this paper includes an old announcement by Seb\"o~\cite{DBLP:journals/jct/Sebo90}, that is, 
a generalization of the classical Kotzig-Lov\'asz decomposition for $T$-joins.%
\footnote{As with the relationship between the classical and general Kotzig-Lov\'asz decompositions 
for perfect matchings, 
the general Kotzig-Lov\'asz decomposition for $T$-joins 
is not a trivial extension  but a refinement of the $T$-join analogue of the classical one. } 
Considering the analogical relationship between perfect matchings and $T$-joins, 
we believe that various consequences will be obtained from our result, 
such as $T$-join analogues of those results derived from the classical or general Kotzig-Lov\'asz decompositions or the basilica decomposition theory.

The reminder of this paper is organized as follows: 
In Section~\ref{sec:definition},  we introduce the basic notation and definitions 
used in this paper. 
In Section~\ref{sec:1fackl}, 
we explain the exact statement of the original general Kotzig-Lov\'asz decomposition for graphs with perfect matchings. In Section~\ref{sec:dist}, 
we introduce preliminary concepts and facts regarding distances in grafts. 
In Section~\ref{sec:elemdist}, 
we provide observation for {\em factor-connected grafts} regarding distances 
to be used in Section~\ref{sec:tjoinkl}. 
In Section~\ref{sec:tjoinkl}, 
we prove the main result of this paper, the analogue of the general Kotzig-Lov\'asz decomposition for minimum joins in grafts.

\section{Definitions} \label{sec:definition}
\subsection{Notation} 
For basic notation, we mostly follow Schrijver~\cite{schrijver2003}. 
We list in the following exceptional or non-standard definitions. 
The symmetric difference of two sets $A$ and $B$, 
that is, $(A\setminus B) \cup (B\setminus A)$ is denoted by $A\triangle B$. 
As usual, a singleton $\{x\}$ is often denoted simply by $x$. 
We treat paths and circuits as graphs. 
That is, a circuit is a connected graph in which every vertex is of degree two. 
A path is a connected graph in which every vertex is of degree no more than two, 
and at least one vertex is of degree less than two. 
A graph with a single vertex and no edges is regarded as a path. 
Let $G$ be a graph.  
The vertex set and the edge set of $G$ are denoted by $V(G)$ and $E(G)$, 
respectively.  
The set of connected components of $G$ is denoted by $\mathcal{C}(G)$. 
Let $X\subseteq V(G)$. 
  The subgraph of $G$ induced by $X$ is denoted by $G[X]$.  
The graph $G[V(G)\setminus X]$ is denoted by $G - X$. 
The contraction of $G$ by $X$ is denoted by $G/X$. 
The set of edges joining $X$ and $V(G)\setminus X$ is denoted by $\parcut{G}{X}$. 
The set of edges that span $X$, that is, 
those with both ends in $X$, are denoted by $E_G[X]$. 
Given $F\subseteq E(G)$, $G-F$ denotes the subgraph of $G$ 
obtained by deleting $F$ without removing any vertices.  
Given a set of disjoint subgraphs $\mathcal{K}$ of $G$, 
$G/\mathcal{K}$ denotes the graph 
obtained by contracting each $K\in\mathcal{K}$ into a  vertex. 
That is, if we let $\mathcal{K} = \{ K_1, \ldots, K_l\}$, where $l\ge 1$, 
then $G/\mathcal{K} = G/V(K_1)/\cdots/V(K_l)$. 
For simplicity, 
we identify the vertices or edges of  $G/\mathcal{K}$ 
with the corresponding items of $G$. 
For example, 
if $e\in E(G)$ is an edge joining $K_1$ and $K_2$, 
we also denote by $e$  the corresponding edge of $G/\mathcal{K}$ 
joining $[K_1]$ and $[K_2]$, where $[K_1]$ and $[K_2]$ are the contracted vertices 
 corresponding to $K_1$ and $K_2$, respectively.

\subsection{Matchings} 
Let $G$ be a graph. 
A  set $M \subseteq E(G)$ is a {\em matching} 
if  each $v\in V(G)$ satisfies $|\parcut{G}{v} \cap M | \le 1$.   
A matching is {\em maximum} if it has the maximum number of edges. 
A matching $M$ is {\em perfect} if 
$|\parcut{G}{v} \cap M | = 1$ for each $v\in V(G)$.  
A perfect matching is also referred to as a {\em $1$-factor}. 
A $1$-factor is a maximum matching, however the converse does not hold. 
A graph is said to be {\em factorizable} if it has a $1$-factor. 

\subsection{Grafts and Joins} 

Let $(G, T)$ be a pair of a graph $G$ and a set $T\subseteq V(G)$.  
A {\em join} of  $(G, T)$ is a set $F\subseteq E(G)$ 
such that $|\parcut{G}{v}\cap F |$ is odd for each $v\in T$, 
and is even for each $v \in V(G)\setminus T$. 
The pair $(G, T)$ is called a {\em graft} if 
each connected component of $G$ has an even number of vertices from $T$. 
When discussing a graft $(G, T)$, 
we often treat the items or properties of $G$ as if they are from $(G, T)$. 
For example, we call an edge of $G$ an edge of $(G, T)$.  
The following statement is classical and 
can be confirmed rather easily by parity arguments; 
see Lov\'asz and Plummer~\cite{lp1986} or Schrijver~\cite{schrijver2003}.  
\begin{fact} \label{fact:join} 
The pair $(G, T)$ has a join if and only if 
  it is a graft. 
\end{fact}

A {\em minimum} join is a join with the minimum number of edges. 
We denote by $\nu(G, T)$ the number of edges in a minimum join of a graft $(G, T)$. 

\begin{remark} 
A join of a graft $(G, T)$ is often referred to as 
a {\em $T$-join} of $G$. 
\end{remark} 

\subsection{Relationship between $1$-Factors and Joins} 

\begin{observation} \label{obs:factorizable}  
If $G$ is a factorizable graph and $T= V(G)$, 
then any $F\subseteq E(G)$ is a minimum join of the graft $(G, T)$ 
if and only if $F$ is a $1$-factor of $G$. 
\end{observation} 

That is, minimum joins of grafts 
are a generalization of $1$-factors of factorizable graphs.

\section{General Kotzig-Lov\'asz Decomposition for $1$-Factors} \label{sec:1fackl}
In this section, 
we explain the original general Kotzig-Lov\'asz decomposition for $1$-factors.  
An  edge  from a factorizable graph is {\em allowed}
 if it can be contained in a $1$-factor. 
Two vertices $u$ and $v$ are said to be {\em factor-connected} 
if there is a path between $u$ and $v$ in which every edge is allowed. 
A factorizable graph is said to be {\em factor-connected} 
if every two vertices are factor-connected. 
A {\em factor-connected component} or {\em factor-component}  
is a maximal factor-connected subgraph.  
Hence, a factorizable graph 
 consists of its factor-components, which are  disjoint, 
and  edges joining distinct factor-components, which are non-allowed. 
The concept of factor-components can be defined in an alternative manner as follows. 
Let $G$ be a factorizable graph. 
Let $\hat{M}$ be the set of allowed edges of $G$. 
A factor-component is a subgraph of $G$ 
induced by $V(C)$, where $C$ is a connected component of a subgraph of $G$ 
determined by $\hat{M}$.

\begin{definition} 
Let $G$ be a factorizable graph. 
A binary relation $\gsim{G}$ over $V(G)$ is defined as follows: 
for $u,v\in V(G)$, 
$u\gsim{G} v$ if $u$ and $v$ are factor-connected and $G-u-v$ is not factorizable. 
\end{definition} 

\begin{theorem}[Kita~\cite{kitacathedral, DBLP:journals/corr/abs-1205-3816, DBLP:conf/isaac/Kita12};  Kotzig~\cite{kotzig1959a, kotzig1959b, kotzig1960}, Lov\'asz~\cite{lovasz1972b}] \label{thm:generalizedkl} 
For any factorizable graph $G$, 
the binary relation $\gsim{G}$ is an equivalence relation over $V(G)$. 
\end{theorem} 

The family of equivalence classes by $\gsim{G}$ is called 
the {\em general Kotzig-Lov\'asz decomposition} 
or simply the {\em Kotzig-Lov\'asz decomposition} of a factorizable graph $G$. 
The restricted statement of Theorem~\ref{thm:generalizedkl} in which $G$ is a factor-connected graph 
was found by Kotzig~\cite{kotzig1959a, kotzig1959b, kotzig1960} and Lov\'asz~\cite{lovasz1972b}. 
In this case, we often call the structure the {\em classical Kotzig-Lov\'asz decomposition}. 
Although a factorizable graph is made up of factor-components, 
the general Kotizg-Lov\'asz decomposition of a factorizable graph $G$ 
is not the mere disjoint union of 
the classical Kotzig-Lov\'asz decomposition of each factor-component of $G$. 
As will be observed in Section~\ref{sec:tjoinkl},  
the first one is, in general, a refinement of the second one.

\section{Preliminaries on Distances} \label{sec:dist}

\subsection{Fundamentals of Distances} 
In this section, 
we introduce the concept of the distance between two vertices in a graft 
 where the edge weight is determined by a given minimum join 
 and explain its basic properties. 
Let $(G, T)$ be a graft, and let $F$ be a minimum join of $(G, T)$.

\begin{definition} 
For each edge $e\in E(G)$, let $w_F(e) := -1$ if $e\in F$,   
and let $w_F(e) := 1$ otherwise. 
Given a subgraph $C$ of $G$, which is typically a path or circuit, 
$w_F(C)$ denotes $\Sigma_{e\in E(C)} w_F(e)$. 
For $u, v\in V(G)$, 
the {\em distance} between $u$ and $v$ in $(G, T)$ regarding $F$ 
is the minimum value of $w_F(P)$, where $P$ is taken over all paths between $u$ and $v$, 
and is denoted by $\distgtf{G}{T}{F}{u}{v}$. 
Note that if $u = v$, then $\distgtf{G}{T}{F}{u}{v} = 0$. 
\end{definition}

The distance between two vertices might be defined on the basis of trails instead of paths. 
However, the next statement shows that if $F$ is a minimum join, 
the concepts of distances defined by paths and trails coincide. 

\begin{proposition}
Let $(G, T)$ be a graft, and let $F\subseteq E(G)$. 
Then,  $F$ is a minimum join of $(G, T)$ 
if and only if $w_F(C) \ge 0$ holds for every circuit $C$. 
\end{proposition}

The next lemma states that 
the distances between two vertices 
does not depend on which minimum join is given.

\begin{lemma}[Seb\"{o}~\cite{DBLP:journals/jct/Sebo90}] \label{lem:dist2nu} 
Let $(G, T)$ be a graft and $F$ be a minimum join of $(G, T)$.  
Then, for any $x, y\in V(G)$ with $x\neq y$, 
$\distgtf{G}{T}{F}{x}{y} = \nu(G, T\triangle \{x, y\}) - \nu(G, T)$. 
\end{lemma} 

Under Lemma~\ref{lem:dist2nu}, 
we abbreviate $\distgtf{G}{T}{F}{x}{y}$ 
to $\distgt{G}{T}{x}{y}$ for any $x, y\in V(G)$.

\subsection{Comb-Bipartite Grafts}

In this section, we introduce the concept of {\em comb-bipartite grafts} 
and their basic properties. 

\begin{definition} 
We say that a graft $(G, T)$ is {\em bipartite} 
if $G$ is a bipartite graph. 
We call $A$ and $B$ {\em color classes} of a bipartite graft $(G, T)$ 
if $A$ and $B$ are color classes of $G$. 
We say that a graft $(G, T)$ is {\em comb-bipartite} 
if $G$ is a bipartite graft with color classes $A$ and $B$, 
the color class $B$ is a subset of $T$, and $\nu(G,T) = |B|$. 
Here, we call $A$ and $B$ the {\em spine} and {\em tooth} sets of 
the comb-bipartite graft $(G, T)$, respectively. 
\end{definition}

The notion of comb-bipartite grafts is closely related to  {\em comb-critical towers} 
introduced by Seb\"o~\cite{DBLP:journals/jct/Sebo90}. A {\em tower} is a pair of a connected graph 
and a set of an odd number of vertices.  
We introduce  comb-bipartite grafts so that 
we can discuss the property of  grafts. 

By definition, the next statement about comb-bipartite grafts is easily confirmed. 
\begin{lemma} \label{lem:comb}
The following three properties are equivalent for 
a bipartite graft $(G, T)$ with color classes $A$ and $B$. 
\begin{rmenum} 
\item The graft $(G, T)$ is comb-bipartite with spine set $A$ and tooth set $B$. 
\item For any minimum join $F$ of $(G, T)$, each $v\in B$ satisfies $|\parcut{G}{v}\cap F| = 1$. 
\item For some minimum join $F$ of $(G, T)$, each $v\in B$ satisfies $|\parcut{G}{v}\cap F| = 1$. 
\end{rmenum} 
\end{lemma}

Accordingly, the next lemma follows. 

\begin{lemma} \label{lem:combpath} 
Let $(G, T)$ be a comb-bipartite graft with spine set $A$ and tooth set $B$, 
and let $F$ be a minimum join of $(G, T)$. 
If $P$ is a path with ends $s\in A$ and $t\in B$ with $w_F(P) = -1$, 
then $ |\parcut{P}{v}\cap F | = 1$ holds for any $v\in V(P)\cap B$. 
\end{lemma}

\subsection{Seb\"{o}'s Distance Decomposition} 

In this section, we present a known statement about distances 
that is taken from a profound theorem by Seb\"o~\cite{DBLP:journals/jct/Sebo90}. 
Given a specific vertex $r$ in a graft, 
we can determine a partition of the vertex set according 
to the distance from $r$. 
Seb\"o provided the property of distances by revealing the structure of this partition.  

\begin{definition} 
Let $(G, T)$ be a graft and $F$ be a minimum join of $(G,T)$. 
Let $r\in V(G)$. 
We define $\levelr{0}{r} := \{ x\in V(G): \distgt{G}{T}{r}{x} = 0 \}$. 
We also define 
 $\laysigr{-}{r} := \{ x \in V(G): \distgt{G}{T}{r}{x} < 0 \}$ 
 and $\layler{0}{r} := \levelr{0}{r} \dot\cup \laysigr{-}{r}$. 
We denote by $Q_r$  the connected component of $G[\layler{0}{r}] - E[\levelr{0}{r}]$ that contains $r$. 
Further, we denote by $Q_r'$ the  graph obtained from $Q_r$ 
by deleting all edges from $E[\levelr{0}{r}]$ and contracting 
each connected component of $Q_r - \levelr{0}{r}$ into one vertex,  
namely, $Q_r' := Q_r/\conn{Q_r - \levelr{0}{r}} - E[\levelr{0}{r}]$. 
For each $K \in \conn{Q_r - \levelr{0}{r}}$,  
 let $[K]$ be the contracted vertices of $Q_r'$ that correspond to $K$, 
and let $T'_r = (T\cap V(Q_r)\cap U_0) \cup \{[K]: K\in \conn{Q_r - \levelr{0}{r}} \}$.  
Note that $(Q_r', T_r')$ is a bipartite graft. 
\end{definition}

The next theorem  is a part of the main result obtained by Seb\"{o}~\cite{DBLP:journals/jct/Sebo90}.

\begin{theorem}[Seb\"o~\cite{DBLP:journals/jct/Sebo90}] \label{thm:sebo} 
Let $(G, T)$ be a graft and $F$ be a minimum join of $(G,T)$. 
Let $r\in V(G)$. 
\begin{rmenum} 
\item \label{item:sebo:separating} Then, no edges in $\parcut{G}{Q_r}$ are allowed in $(G, T)$. 
\item \label{item:sebo:extreme} No edges of $Q_r[\levelr{0}{r}]$ are allowed in $(G, T)$. 
\item \label{item:sebo:join2comb}  For each $K \in  \conn{Q_r - \levelr{0}{r}}$, $|\parcut{G}{K} \cap F | = 1$; 
let $r_K \in V(K)$ and $s_K\in \levelr{0}{r}$ be the vertices such that $\{s_Kr_K \} = \parcut{G}{K} \cap F$. 
\item \label{item:sebo:comb} 
The graft $(Q_r', T'_r)$ is  comb-bipartite, 
whose tooth set is $\{[K]: K\in   \conn{Q_r - \levelr{0}{r}} \}$, 
and  $\{ s_Kr_K : K\in    \conn{Q_r - \levelr{0}{r}} \}$ forms a minimum join of $(Q_r', T'_r)$. 
\item \label{item:sebo:join2negcomp} For each $K\in  \conn{Q_r - \levelr{0}{r}}$, 
 $F\cap E(K)$ is a minimum join of the graft $(K, ({T\cap V(K)})\triangle \{r_K\})$; and, 
\item \label{item:sebo:negcomp2dist} 
for any $x\in V(K)$, 
$\distgtf{K}{(T\cap V(K)) \triangle \{ r_K \}}{F\cap E(K)}{x}{r_K} \le 0$.  
\end{rmenum} 
\end{theorem} 

In this paper, we refer to the above-mentioned structure of $(G, T)$ 
as {\em Seb\"{o}'s distance decomposition with the root $r$}.

\section{Factor-Connectivity and Distance in Grafts} \label{sec:elemdist}

In this section, 
we introduce  a graft analogue of  factor-connectivity and show some basic properties 
of {\em factor-connected grafts} 
regarding distances to be used for proving our main theorem.

Let $(G, T)$ be a graft. 
An edge $e\in E(G)$ is {\em allowed} 
if there is a minimum join of $(G, T)$ that contains $e$. 
We say that vertices $u,v\in V(G)$ are {\em factor-connected} 
if $(G, T)$ has a path whose edges are allowed. 
We say that a graft is {\em factor-connected}  
if any two vertices are factor-connected. 
We call a maximal factor-connected subgraph of $G$ a {\em factor-connected component} 
or {\em factor-component}, in short,  of $(G, T)$. 
We denote the set of factor-components of $(G, T)$ by $\tcomp{G}{T}$. 
It can easily be  observed from the definition that 
$G$ consists of  factor-components, 
which are mutually disjoint,  
and edges joining distinct factor-components, which are not allowed.

We now present some observations about the distance between two factor-connected vertices. 
Theorem~\ref{thm:sebo} \ref{item:sebo:separating} implies the next lemma rather immediately. 

\begin{lemma} \label{lem:elem2dist} 
If $u$ and $v$ are factor-connected vertices in a graft $(G, T)$, 
then $\distgt{G}{T}{u}{v} \le 0$ holds.  
\end{lemma} 
\begin{proof}  
Consider Seb\"{o}'s distance decomposition with the root $u$. 
From Theorem~\ref{thm:sebo} \ref{item:sebo:separating}, 
we have $u, v\in V(Q_u)$. 
As $V(Q_u) \subseteq \layler{0}{r}$ holds, the statement follows. 
\end{proof}

From Lemmas~\ref{lem:combpath} and \ref{lem:elem2dist}, 
the next lemma is easily confirmed. 

\begin{lemma} \label{lem:elemcomb2dist} 
If $(G, T)$ is a factor-connected comb-bipartite  graft with tooth set $B$ and spine set $A$, 
then, $\distgt{G}{T}{x}{y}  = -1$ for any $x\in A$ and any $y\in B$. 
\end{lemma}

\section{General Kotzig-Lov\'asz Decomposition for  Joins in Grafts} \label{sec:tjoinkl} 

In this section, we prove our main result  in Theorem~\ref{thm:tjoinkl}, 
a generalization of the general Kotzig-Lov\'asz decomposition for grafts. 

\begin{definition} 
Let $(G, T)$ be a graft. 
For $u, v\in V(G)$, 
we say that $u\gtsim{G}{T} v$ if $u$ and $v$ are identical 
or if $u$ and $v$ are distinct  factor-connected vertices such that $\nu(G, T\triangle \{u, v\})) = \nu(G, T)$. 
\end{definition} 

The binary relation $\gtsim{G}{T}$ is a generalization of 
the equivalence relation $\gsim{G}$; 
confirm the following by Observation~\ref{obs:factorizable}: 
\begin{observation} \label{obs:reduction} 
Let $(G, T)$ be a graft in which  
$G$ is a factorizable graph and the set $T$ is  $V(G)$. 
Then, for any $u, v\in V(G)$, 
 $u \gsim{G} v$ holds if and only if  $u \gtsim{G}{T} v$ holds. 
\end{observation} 

In the following, we prove that $\gtsim{G}{T}$ is an equivalence relation. 
Note that Lemmas~\ref{lem:dist2nu} and \ref{lem:elem2dist} imply the following: 
\begin{lemma} \label{lem:nonsim} 
Let $(G, T)$ be a graft, and let $u,v\in V(G)$ be factor-connected vertices.  
Then, $u \gtsim{G}{T} v$ holds if and only if  $\distgt{G}{T}{u}{v} = 0$; 
by contrast, $u \gtsim{G}{T} v$ does not hold 
if and only if $\distgt{G}{T}{u}{v}  < 0$. 
\end{lemma}

The next two lemmas are provided for Theorem~\ref{thm:tjoinkl}.

\begin{lemma} \label{lem:compimage} 
Let $(G, T)$ be a graft, let $r\in V(G)$, and let $H\in\tcomp{G}{T}$ be the factor-component with $r\in V(H)$. 
Then, $V(H)\subseteq Q_r$ holds. 
In the graft $(Q_r', T_r')$,  
the  vertices in 
$( V(H)\cap \levelr{0}{r} ) \cup \{ [K] : K \in\conn{Q_r - \levelr{0}{r}} \mbox{ and } V(K)\cap V(H)\}$, 
that is, the vertices from the subgraph of $Q_r'$ that corresponds to $H$, 
are factor-connected. 
\end{lemma} 
\begin{proof} 
From Theorem~\ref{thm:sebo} \ref{item:sebo:separating},  $V(H)\subseteq Q_r$ follows.  
From Theorem~\ref{thm:sebo} \ref{item:sebo:extreme}, 
for any two vertices of $V(H)$, 
there is a path whose edges are allowed edges of $(G, T)$ from $E(H)\setminus E[\levelr{0}{r}]$. 
From Theorem~\ref{thm:sebo} \ref{item:sebo:join2comb}, 
the allowed edges of $(G, T)$ that join $\levelr{0}{r}$ and   components from $\conn{Q_r - \levelr{0}{r}}$
are also allowed in $(Q_r', T_r')$. 
Therefore, the statement follows. 
\end{proof}

\begin{lemma} \label{lem:pathext} 
Let $(G, T)$ be a graft, let $F$ be a minimum join of $(G, T)$, and let $r\in V(G)$. 
Let $u\in \levelr{0}{r}$ and $K\in\conn{Q_r - \levelr{0}{r}}$.  
If the graft $(Q_r', T'_r)$ has a path $P$ between $u$ and $[K]$ 
with $w_F(P) = -1$, 
then, for any $v\in V(K)$, 
the graft $(G, T)$ has a path $\hat{P}$ between $u$ and $v$ 
with $w_F(\hat{P}) \le -1$. 
\end{lemma} 
\begin{proof} 
First, note that, according to Theorem~\ref{thm:sebo} \ref{item:sebo:join2comb}, 
$F$ contains a minimum join of $(Q_r', T_r')$. 
Under Lemma~\ref{lem:combpath}, 
for each $L \in \conn{Q_r - \levelr{0}{r}}$ with $[L]\in V(P)$,  
define the following: 
Let $e_L$ be the edge from $\parcut{P}{[L]}\cap F$; 
  additionally, let $x_L \in V(G)$ be the end of $e_L$  from $V(L)$ 
 by regarding $e_L$ as an edge of $G$. 
Furthermore, if $L \neq K$, then   
let $f_L$ be the edge from $\parcut{P}{[L]}\setminus F$; 
additionally, let $y_L \in V(G)$ be the end of $f_L$ from $V(L)$ 
by regarding $f_L$ as an edge of $G$. 
If $L = K$, then let $y_L$ be a given vertex $v\in V(K)$. 
By Theorem~\ref{thm:sebo} \ref{item:sebo:negcomp2dist}, 
each $L$ has a path $R_L$ between $x_L$ and $y_L$ with $w(R_L) \le 0$. 
By replacing each $[L]$ with $Q_L$ over $P$, 
we obtain a desired path $\hat{P}$ with $w_F(\hat{P}) \le -1$. 
\end{proof}

We now prove Theorem~\ref{thm:tjoinkl}. 

\begin{theorem} \label{thm:tjoinkl}
For any graft $(G, T)$, 
the binary relation $\gtsim{G}{T}$ is an equivalence relation on $V(G)$. 
\end{theorem}
\begin{proof} 
Symmetry and reflexivity are obvious from the definition. 
Hence, we prove transitivity in the following. 
Let $u, v, w\in V(G)$ be such that $u \gtsim{G}{T} v$ and $v \gtsim{G}{T} w$. 
If any two from $u$, $v$, and $w$ are identical, then the statement obviously holds. 
Therefore, assume that $u, v, w$ are pairwise distinct, and suppose, to the contrary, that $u\gtsim{G}{T} w$ does not hold. 
Let $H$ be the factor-component that contains $u$, $v$, and $w$, 
and let $F$ be a minimum join of $(G, T)$.  
Consider Seb\"{o}'s distance decomposition with the root $u$. 
Because $u\gtsim{G}{T} v$ is assumed, Lemma~\ref{lem:nonsim} implies $v\in \levelr{0}{u} \cap V(Q_u)$. 
Under the present supposition, 
Lemma~\ref{lem:nonsim} implies  $w\in \laysigr{-}{r}$,  and  accordingly, 
there exists $K\in\conn{Q_u - \levelr{0}{u}}$ such that $w\in V(K)$. 
According to Lemma~\ref{lem:compimage}, 
in the comb-bipartite graft $(Q_u', T'_r)$,  
the three vertices $u$, $v$, and $[K]$ are factor-connected. 
Therefore, from Lemma~\ref{lem:elemcomb2dist}, 
$Q_u'$ has a path $P$ between $v$ and $[K]$ with $w_F(P) = -1$. 
Thus, from Lemma~\ref{lem:pathext}, 
$G$ has a path $\hat{P}$ between $v$ and $w$ with $w_F(\hat{P}) \le -1$. 
This contradicts $v \gtsim{G}{T} w$. 
Thus, $u \gtsim{G}{T} w$ is proved. 
\end{proof} 

We call the family of equivalence classes determined by $\gtsim{G}{T}$
 the {\em general Kotzig-Lov\'asz decomposition} of a graft $(G, T)$. 
We denote this family by $\gtpart{G}{T}$. 
From Observation~\ref{obs:reduction}, 
the general Kotzig-Lov\'asz decomposition for general grafts 
is a generalization of the general Kotzig-Lov\'asz decomposition for factorizable graphs. 
Our result also includes, as a special case, 
the result announced in Seb\"o~\cite{DBLP:journals/jct/Sebo90}; 
if we restrict $(G, T)$ to be factor-connected, 
then $\gtpart{G}{T}$ is a graft analogue of the classical Kotzig-Lov\'asz decomposition.

By the definition of $\gtsim{G}{T}$, 
each equivalence class is contained in the vertex set of some factor-component of $(G, T)$. 
Therefore, for each $H\in\tcomp{G}{T}$, 
the family $\{ S\in \gtpart{G}{T}: S \subseteq  V(H) \}$  
forms a partition of $V(H)$. 
We denote this family by $\pargtpart{G}{T}{H}$. 
However, note that, as is also the case for $1$-factors, 
$\gtpart{G}{T}$ is not a mere disjoint union of 
$\pargtpart{G}{T}{H}$ taken over every $H\in\tcomp{G}{T}$, 
but has a more refined structure. 

\begin{observation} 
Let $(G, T)$ be a graft, and let $H\in\tcomp{G}{T}$. 
Then, the family $\pargtpart{G}{T}{H}$ 
is a refinement of $\gtpart{H}{T\cap V(H)}$. 
That is, 
if $u,v\in V(H)$ satisfy $u\gtsim{G}{T} v$, 
then $u \gtsim{H}{T\cap V(H)} v$ holds; 
however, the converse does not hold in general. 
\end{observation}

\begin{figure}
\includegraphics[width=.7\textwidth]{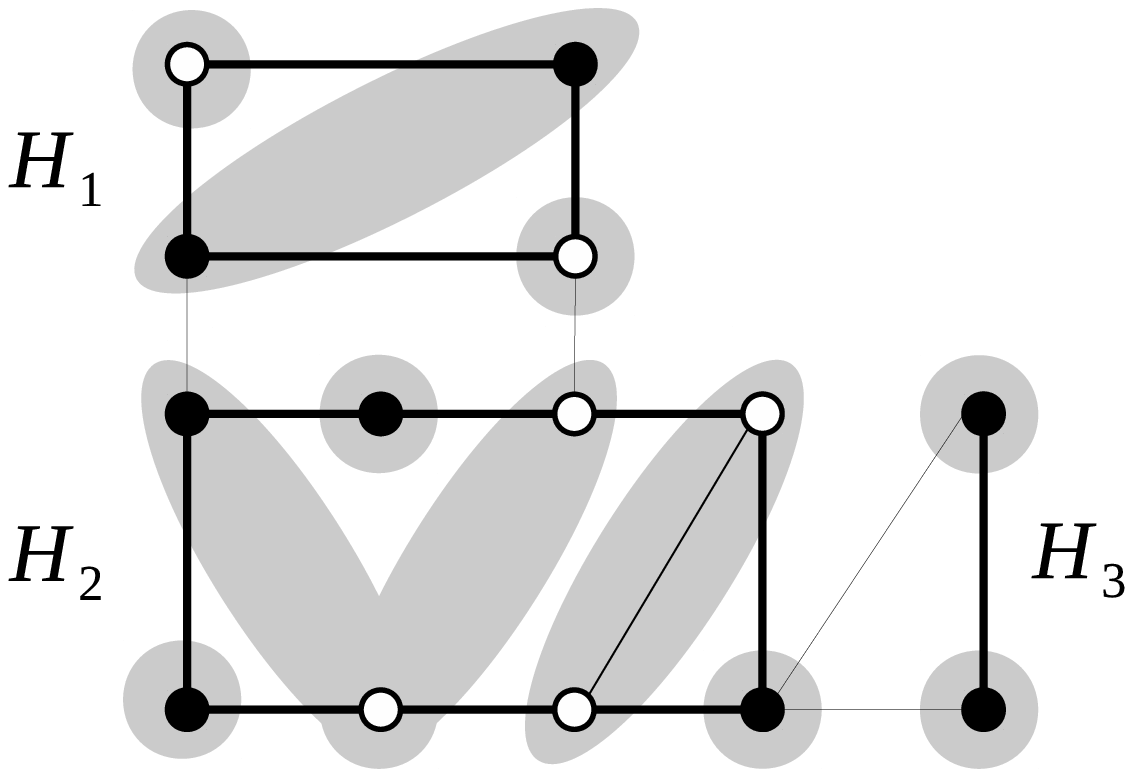} 
\caption{Example of a graft $(G, T)$ and its general Kotzig-Lov\'asz decomposition. 
The black and white vertices represent vertices in $T$ and not in $T$, respectively. 
The thick edges represent allowed edges, and  $\tcomp{G}{T} = \{H_1, H_2, H_3\}$. 
Each equivalence class from $\gtpart{G}{T}$ is indicated by a gray region. } 
\label{fig:kl} 

\includegraphics[width=.3\textwidth]{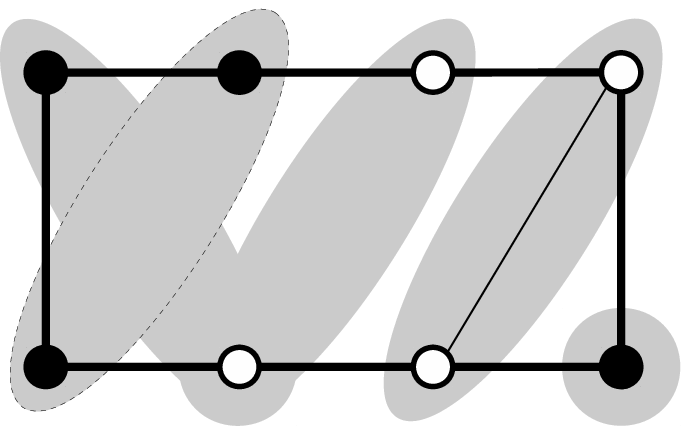} 
\caption{The Kotzig-Lov\'asz decomposition of the subgraft $(H_1, T\cap V(H_1))$. 
The family $\pargtpart{G}{T}{H_1}$ from Figure~\ref{fig:kl} is a proper refinement of $\gtpart{H_1}{T\cap V(H_1)}$.
} 
\label{fig:elemkl}  
\end{figure} 

\begin{example} 
For example, consider the graft $(G, T)$ given in Figure~\ref{fig:kl}. 
For this graft, $\gtpart{G}{T}$ has 10 members as follows. 
For each factor-component $H_1$, $H_2$, and $H_3$, 
$\pargtpart{G}{T}{{H_1}}$, $\pargtpart{G}{T}{{H_2}}$, and $\pargtpart{G}{T}{{H_3}}$ 
have three, five, and two members, respectively. 
If we consider $(H_1, T\cap V(H_1))$ as a single graft, 
$\gtpart{H_1}{T\cap V(H_1)}$ has four members as shown in Figure~\ref{fig:elemkl}. 
The family $\pargtpart{G}{T}{{H_1}}$ given in Figure~\ref{fig:kl} 
is a proper refinement of $\gtpart{H_1}{T\cap V(H_1)}$.  
\end{example}

\begin{ac} 
This study is partly supported by JSPS KAKENHI 15J09683. 
\end{ac} 

\bibliographystyle{splncs03.bst}
\bibliography{tjoinkl.bib}

\end{document}